\documentclass{amsart}
\usepackage{latexsym,amssymb,amsmath,amsthm,amscd,graphicx}

\usepackage{esint} 
\usepackage{color}

\numberwithin{equation}{section}
\newtheorem{theorem}{Theorem}[section]
\newtheorem{lemma}[theorem]{Lemma}
\newtheorem{corollary}[theorem]{Corollary}

\theoremstyle{definition}

\newtheorem{definition}[theorem]{Definition}
\newtheorem{example}[theorem]{Example}

\theoremstyle{remark}

\usepackage{url,hyperref}

\usepackage{color}
\definecolor{dblue}{rgb}{0,0,0.45}
\definecolor{red}{rgb}{0.7,0,0}

\begin{document}
%
\title[Paraproduct in Besov--Morrey spaces]{Paraproduct in Besov--Morrey spaces}
%
%
\author{Yoshihiro Sawano}
%
%
%

\maketitle              

\begin{abstract}
Recently it turned out that the paraproduct
plays the key role in some highly singular partial differential equations.
In this note the counterparts for Besov--Morrey spaces
are obtained.
This note is organized in a self-contained manner.
\end{abstract}

\section{Introduction}\label{s1}

In this note we investigate the boundedness property
of the pointwise multiplier of the functions
in H\"{o}lder--Zygmund spaces and 
Besov--Morrey spaces
including the commutators.
Starting from the seminal papers \cite{GIP15,GIP16,Hairer14},
we investigate these operators from the viewpoint of harmonic analysis.

To describe our first result,
we recall some notation.
First, we use the following convention on balls
in ${\mathbf R}^n$
here and below:
We denote by $B(x,r)$
the {\it ball centered at $x$ of radius $r$}.
Namely, we write
\[
B(x,r) \equiv
\{y \in {\mathbf R}^n\,:\, |x-y|<r\}
\]
when $x \in {\mathbf R}^n$
and $r>0$.
Given a ball $B$,
we denote by $c(B)$ its {\it center} 
and by $r(B)$ its {\it radius}.
We write $B(r)$ instead of $B(o,r)$,
where $o\equiv (0,0,\ldots,0)$.
Keeping this definition of balls in mind,
we define Morrey spaces.
Let $1 \le q \le p<\infty$.
Define the {\it Morrey norm}
$\|\cdot\|_{{\mathcal M}^p_q}$ by
\[
\| f \|_{{\mathcal M}^p_q}
\equiv
\sup_{x \in {\mathbf R}^n, r>0}
|B(x,r)|^{\frac{1}{p}-\frac{1}{q}}\|f\|_{L^q(B(x,r))}
\]
for a measurable function $f$.
The {\it Morrey space}
${\mathcal M}^p_q({\mathbf R}^n)$
is the set of all the measurable functions $f$
for which
$\| f \|_{{\mathcal M}^p_q}$
is finite.
We move on to the definition of Besov--Morrey spaces.
Choose $\psi \in C^\infty_{\rm c}({\mathbf R}^n)$
so that
\begin{equation}\label{eq:180812-3}
\chi_{B(\frac65)} \le \psi \le \chi_{B(\frac32)}.
\end{equation}
We write
\[
\varphi_0(\xi)=\psi(\xi), \quad
\varphi_j(\xi)=\psi(2^{-j}\xi)-\psi(2^{-j+1}\xi), \quad
\psi_j(\xi)=\psi(2^{-j}\xi)
\]
for $j \in {\mathbf N}$ and $\xi \in {\mathbf R}^n$.

For $f \in L^1({\mathbf R}^n)$,
define
the {\it Fourier transform} and
the {\it inverse Fourier transform}
by{\rm:}
\begin{eqnarray*}
{\mathcal F}f(\xi)
\equiv
(2\pi)^{-\frac{n}{2}}
\int_{{\mathbf R}^n} f(x){\rm e}^{-i x \cdot \xi} {\rm d}x, \quad
{\mathcal F}^{-1}f(x)
\equiv
(2\pi)^{-\frac{n}{2}}
\int_{{\mathbf R}^n} f(\xi){\rm e}^{i x \cdot \xi} {\rm d}\xi.
\end{eqnarray*}
Here and below we write
$\theta(D)f\equiv {\mathcal F}^{-1}[\theta \cdot {\mathcal F}f]$
for $\theta \in {\mathcal S}({\mathbf R}^n)$ 
and $f \in {\mathcal S}'({\mathbf R}^n)$.
It is known that
$\theta(D)f \in {\mathcal S}'({\mathbf R}^n)\cap L^1_{\rm loc}({\mathbf R}^n)$
and it satisfies
\[
\theta(D)f(x)=(2\pi)^{-\frac{n}{2}}\langle f,{\mathcal F}^{-1}\theta(x-\cdot) \rangle
\]
for all $x \in {\mathbf R}^n$.
We define
\[
\| f \|_{{\mathcal N}^s_{pqr}}
\equiv
\left(
\sum_{j=0}^\infty 
(2^{j s}\| \varphi_j(D)f \|_{{\mathcal M}^p_q})^r
\right)^{\frac1r}
\]
for $f \in {\mathcal S}'({\mathbf R}^n)$.

Let $1 \le q \le p<\infty$,
$1 \le r\le \infty$
and $s \in \mathbf{R}$.
The space
${\mathcal N}^s_{pqr}({\mathbf R}^n)$,
which we call the {\it Besov--Morrey space},
is the set
of all
$f \in {\mathcal S}'({\mathbf R}^n)$
for which
the norm
$\| f \|_{{\mathcal N}^s_{pqr}}$
is
finite.
The parameter $s$ describes the differential property
in terms of Morrey spaces
as is indicated by the relations
${\mathcal N}^{s+\varepsilon}_{pqr}({\mathbf R}^n)
\subset
{\mathcal N}^s_{pqr}({\mathbf R}^n)$
and
$\partial_j:{\mathcal N}^{s+1}_{pqr}({\mathbf R}^n)
\subset
{\mathcal N}^s_{pqr}({\mathbf R}^n)$
for all $\varepsilon>0$ and $j=1,2,\ldots,n$.
It is also clear from the triangle inequality in ${\mathcal M}^{p}_{q}({\mathbf R}^n)$  that
${\mathcal N}^0_{pq1}({\mathbf R}^n) \subset {\mathcal M}^{p}_{q}({\mathbf R}^n)$.
The main results in this note are the following:
\begin{theorem}\label{thm:main1}
Let 
$1 \le q_1 \le p_1<\infty$,
$1 \le q_2 \le p_2<\infty$,
$1 \le q \le p<\infty$,
$1 \le r \le \infty$,
and
$s>0$.
Assume that
\[
\frac1p=\frac{1}{p_1}+\frac{1}{p_2}, \quad
\frac1q=\frac{1}{q_1}+\frac{1}{q_2}.
\]
Then
for $f \in {\mathcal N}^{s}_{p_1q_1r}({\mathbf R}^n)$
and $g \in {\mathcal N}^{s}_{p_2q_2r}({\mathbf R}^n)$
the product
$f \cdot g \in {\mathcal N}^{s}_{p q r}({\mathbf R}^n)$ makes sense
and satisfies
\[
\| f \cdot g\|_{{\mathcal N}^{s}_{p q r}}
\le C
\|f\|_{{\mathcal N}^{s}_{p_1q_1r}}
\|g\|_{{\mathcal N}^{s}_{p_2q_2r}}.
\]
\end{theorem}
Theorem \ref{thm:main1} is an extension of the inequality
\[
\| f \cdot g\|_{{\mathcal M}^{p}_{q}}
\le C
\|f\|_{{\mathcal M}^{p_1}_{q_1}}
\|g\|_{{\mathcal M}^{p_2}_{q_2}}
\]
for 
$f \in {\mathcal M}^{p_1}_{q_1}({\mathbf R}^n)$
and
$g \in {\mathcal M}^{p_2}_{q_2}({\mathbf R}^n)$.
The proof of Theorem \ref{thm:main1} hinges on the paraproduct
introduced by Bony \cite{Bony81}.
Let $f,g \in {\mathcal S}'({\mathbf R}^n)$.
The (right) paraproduct
$f \preceq g$
is defined to be
\[
f \preceq g=\sum_{j=2}^\infty \psi_{j-2}(D)f \cdot \varphi_j(D)g,
\]
while the (left) paraproduct
$f \succeq g$
is defined to be
\[
f \succeq g=\sum_{j=2}^\infty \varphi_j(D)f \cdot \psi_{j-2}(D)g.
\]
Furthermore, the resonant operator
$f \odot g$ is defined by
\[
f \odot g=
\sum_{j=0}^\infty \varphi_j(D)f \cdot \varphi_j(D)g
+
\sum_{j=1}^\infty \varphi_{j-1}(D)f \cdot \varphi_j(D)g
+
\sum_{j=1}^\infty \varphi_j(D)f \cdot \varphi_{j-1}(D)g.
\]
We need some assumptions on $f$ and $g$
to justify these definitions.
These three linear operators are key linear operators
used in the proof of Theorem \ref{thm:main1}.

Another aim of this paper is to extend
the results used in \cite{GIP15,Hairer14},
which also use these operators, to the Morrey setting:
\begin{theorem}\label{thm:main2}
Assume that the parameters $\alpha,\beta,s$ satisfy
\[
0<\alpha \le 1, \quad
s+\beta<0<
s+\alpha+\beta.
\]
Then for $f \in {\rm Lip}_\alpha({\mathbf R}^n)$,
$g \in {\mathcal C}^\beta({\mathbf R}^n)$
and 
$h \in {\mathcal N}^{s}_{p q r}({\mathbf R}^n)$
\[
\|(f \preceq g) \odot h-f(g \odot h)\|_{{\mathcal N}^{s+\alpha+\beta}_{p q r}}
\le C
\|f\|_{{\rm Lip}^\alpha}
\|g\|_{{\mathcal C}^\beta}
\|h\|_{{\mathcal N}^{s}_{p q r}}.
\]
\end{theorem}

This result is a counterpart to \cite[Lemma 2.4]{GIP15}.

Here we briefly recall how 
Besov--Morrey spaces emerged.
See \cite{Sawano18,YSY10-2} for an exhaustive account.
The first paper dates back to 1984.
In \cite{Netrusov84}
Netrusov considered Besov--Morrey spaces.
Later on Kozono and Yamazaki investigated Besov--Morrey spaces
and applied them to the Navier--Stokes equations
\cite{KoYa94}.
Mazzucato expanded this application more
in \cite{Mazzucato02}.
Decompositions of Besov--Morrey spaces
can be found in
\cite{Mazzucato01,SaTa07,TaXu05}.
After that Yang and Yuan defined
Besov-type spaces
and 
Triebel--Lizorkin-type spaces
in \cite{YaYu08-2,YaYu10-1}.
A close relation between these spaces is pointed out
in \cite{SYY10}.
Recently more and more is investigated.
For example, Haroske and Skrzypczak
investigated embedding relation of Besov--Morrrey spaces
\cite{HaSk14}.
One of the important consequence
of definining the Besov--Morrey spaces
is that we have the embedding
\[
{\mathcal N}^{s}_{p q\infty}({\mathbf R}^n)
\hookrightarrow
{\mathcal C}^{s-\frac{n}{p}}({\mathbf R}^n)
\]
for $s>\frac{n}{p}$.
See \cite{SST09}.

We organize this paper as follows:
Section \ref{s2} is devoted to collecting some preliminary facts.
In Section \ref{s3} we prove Theorem \ref{thm:main1}
and 
in Section \ref{s4} we prove Theorem \ref{thm:main2}.
\section{Preliminaries}\label{s2}

\subsection{Schwartz distributions and the Fourier transform}

Let us recall the notation of multi-indexes
to define the Schwartz space ${\mathcal S}({\mathbf R}^n)$.
By \lq \lq a multi-index",
we mean an element in ${\mathbf N}_0{}^n\equiv \{0,1,2,\ldots\}^n$.
In this paper
a tacit understanding is that
all functions assume their value in ${\mathbf C}$.
For a multi-index
$\alpha=(\alpha_1,\alpha_2,\ldots,\alpha_n) \in {\mathbf N}_0{}^n$
$x=(x_1,x_2,\ldots,x_n) \in {\mathbf R}^n$, we define
$x^\alpha \equiv x_1{}^{\alpha_1}x_2{}^{\alpha_2}\cdots x_n{}^{\alpha_n}$.
For
a multi-index
$\beta=(\beta_1,\beta_2,\ldots,\beta_n) \in {\mathbf N}_0{}^n$
and
$f \in C^{\infty}({\mathbf R}^n)$,
we set
\[
\partial^\beta f
\equiv
\left(\frac{\partial}{\partial x_1}\right)^{\beta_1}
\left(\frac{\partial}{\partial x_2}\right)^{\beta_2}
\ldots
\left(\frac{\partial}{\partial x_n}\right)^{\beta_n}f.
\]
\begin{definition}[Schwartz function space ${\mathcal S}({\mathbf R}^n)$]
\label{defi:cS}
For multi-indexes $\alpha,\beta \in {\mathbf N}_0{}^n$
and a function $\varphi$,
write
$\varphi_{(\alpha,\beta)}(x)\equiv x^\alpha \partial^\beta \varphi(x)$,
$x \in {\mathbf R}^n$
temporarily.
The {\it Schwartz function space ${\mathcal S}({\mathbf R}^n)$}
is the set of all the functions satisfying
\[
{\mathcal S}({\mathbf R}^n)\equiv
\bigcap_{\alpha,\beta \in {\mathbf N}_0{}^n}
\left\{\varphi \in C^{\infty}({\mathbf R}^n) \, : \,
\varphi_{(\alpha,\beta)} \in L^\infty({\mathbf R}^n)
\right\}.
\]
The elements in ${\mathcal S}({\mathbf R}^n)$
are called the {\it test functions}.
\index{test functions@test functions}
\index{Schwartz function space@Schwartz function space ${\mathcal S}({\mathbf R}^n)$}
\end{definition}
Denote by
${\mathcal S}'({\mathbf R}^n)$
the set of all continuous linear mappings
from ${\mathcal S}({\mathbf R}^n)$ to ${\mathbf C}$.
Denote by $\langle f,\varphi \rangle$
the value of $f$ evaluated at $\varphi$;
$\langle f,\varphi \rangle \equiv f(\varphi)$.

Note that
${\mathcal S}({\mathbf R}^n)$ is embedded into
$L^1({\mathbf R}^n)$
and that
${\mathcal F}$ mapsto ${\mathcal S}({\mathbf R}^n)$
isomorphically to itself.
Thus by duality
${\mathcal F}$ mapsto ${\mathcal S}'({\mathbf R}^n)$
isomorphically to itself.

A function $h \in C^{\infty}({\mathbf R}^n)$
is said to {\it have at most polynomial growth at infinity},
if for all $\alpha \in {\mathbf N}_0{}^n$,
there exist
$C_\alpha>0$
and
$N_\alpha>0$
such that{\rm:}
\begin{equation}\label{eq:090227-1009}
|\partial^\alpha h(x)| \le
C_\alpha
\langle x \rangle^{N_\alpha}, \quad
x \in {\mathbf R}^n.
\end{equation}
Here we are interested in the inclusion:
\begin{equation}\label{eq:15.13}
{\rm supp}({\mathcal F}[f \cdot g]) \subset {\rm supp}({\mathcal F}f)+{\rm supp}({\mathcal F}g)
\end{equation}
for $f,g \in {\mathcal S}'({\mathbf R}^n)$
having  at most polynomial growth at infinity.
Usually we assume that
${\mathcal F}f$ is compactly supported.

Let
$\Omega$
be a bounded set in ${\mathbf R}^n$.
Denote 
by
${\mathcal S}'_\Omega({\mathbf R}^n)$
the set of all distributions
whose Fourier transform is contained
in the closure $\overline{\Omega}$.
Define
${\mathcal S}_\Omega({\mathbf R}^n)
\equiv
{\mathcal S}'_\Omega({\mathbf R}^n)
\cap {\mathcal S}({\mathbf R}^n)$.

\begin{lemma}\label{lem:180812-2}
\
\begin{enumerate}
\item
For all 
$F\in C^\infty_{\rm c}({\mathbf R}^n)$,
$G \in {\mathcal S}({\mathbf R}^n)$,
\begin{equation}\label{eq:180612-1}
{\rm supp}(F*G) \subset {\rm supp}(F)+{\rm supp}(G).
\end{equation}
\item
Let $K$ be a compact set.
Then
for all 
$f\in {\mathcal S}_K({\mathbf R}^n)$,
$g \in {\mathcal S}({\mathbf R}^n)$,
$(\ref{eq:15.13})$ holds.
\end{enumerate}
\end{lemma}

\begin{proof}
\
\begin{enumerate}
\item
The proof of (\ref{eq:180612-1}) is standard:
Simply write out the convolution $f*g$
in full in terms of the integral to have
\begin{eqnarray*}
\{x \in {\mathbf R}^n\,:\,F*G(x) \ne 0\}
&\subset&
\{x \in {\mathbf R}^n\,:\,F(x) \ne 0\}
+
\{x \in {\mathbf R}^n\,:\,G(x) \ne 0\}\\
&\subset&
{\rm supp}(F)+{\rm supp}(G).
\end{eqnarray*}
Since ${\rm supp}(F)$ is compact and ${\rm supp}(G)$
is closed,
${\rm supp}(F)+{\rm supp}(G)$
is a closed set.
Thus, taking the closure of the above inclusion,
we conclude that (\ref{eq:180612-1}) holds.
\item
Inclusion
$(\ref{eq:15.13})$
is a consequence of
${\mathcal F}[f \cdot g]=(2\pi)^{-\frac{n}{2}}{\mathcal F}f * {\mathcal F}g$
and the fact that
${\mathcal F}$ maps ${\mathcal S}({\mathbf R}^n)$
isomorphically.
\end{enumerate}
\end{proof}

Define the convolution
\index{convolution@convolution}
$f*g$ by
$\displaystyle
f*g(x)\equiv 
\int_{{\mathbf R}^n}f(x-y)g(y){\rm d}y
$
as long as the integral makes sense.

A band-limited distribution is a
distribution whose
Fourier transform is compactly supported.
\begin{lemma}\label{lem:180812-3}
For all band-limited distributions
$f \in {\mathcal S}'({\mathbf R}^n)$
and all functions
$g \in {\mathcal S}({\mathbf R}^n)$,
$(\ref{eq:15.13})$ holds.
\end{lemma}

\begin{proof}
Let $\tau \in C^\infty_{\rm c}({\mathbf R}^n)$
be such that
${\rm supp}(\tau) \cap ({\rm supp}({\mathcal F}f)+{\rm supp}({\mathcal F}g))= \emptyset$.
We need to show that
\[
\langle {\mathcal F}[f \cdot g],\tau\rangle
=0.
\]
By the definition of the Fourier transform
this amounts to showing:
\[
\langle f \cdot g,{\mathcal F}\tau\rangle
=0.
\]
Since $g \in {\mathcal S}({\mathbf R}^n)$,
we have
\[
\langle f \cdot g,{\mathcal F}\tau\rangle
=
\langle f,g \cdot {\mathcal F}\tau\rangle
\]
from the definition of the pointwise multiplication
$f \cdot g \in {\mathcal S}'({\mathbf R}^n)$
for
$f \in {\mathcal S}'({\mathbf R}^n)$
and 
$g \in {\mathcal S}({\mathbf R}^n)$.
We note that
\[
{\mathcal F}^{-1}[g \cdot {\mathcal F}\tau]
=
(2\pi)^{-\frac{n}{2}}{\mathcal F}^{-1}g*\tau.
\]
Thus,
by the definition of the Fourier transform
${\mathcal F}$ acting on 
${\mathcal S}'({\mathbf R}^n)$
\[
\langle f \cdot g,{\mathcal F}\tau\rangle
=
(2\pi)^{-\frac{n}{2}}
\langle {\mathcal F}f,{\mathcal F}^{-1}g*\tau \rangle.
\]
From the definition of the Fourier transform
$x \in {\rm supp}({\mathcal F}g)$
if and only if
$-x \in {\rm supp}({\mathcal F}^{-1}g)$.
Since
${\rm supp}(\tau) \cap ({\rm supp}({\mathcal F}f)+{\rm supp}({\mathcal F}g))= \emptyset$,
we have
\[
{\rm supp}(\tau*{\mathcal F}^{-1}g) \cap  {\rm supp}(f)
\subset
({\rm supp}(\tau)+{\rm supp}({\mathcal F}^{-1}g)) \cap {\rm supp}(f) =\emptyset.
\]
thanks to Lemma \ref{lem:180812-2}.
Thus,
$
\langle f \cdot g,{\mathcal F}\tau\rangle=0$
and
$(\ref{eq:15.13})$ holds.
\end{proof}

\begin{corollary}\label{cor:180812-1}
For all band-limited $f,g \in {\mathcal S}({\mathbf R}^n)$,
$(\ref{eq:15.13})$ holds.
\end{corollary}

\begin{proof}
Let $\tau \in C^\infty_{\rm c}({\mathbf R}^n)$
be such that
${\rm supp}(\tau) \cap ({\rm supp}({\mathcal F}f)+{\rm supp}({\mathcal F}g))=\emptyset$.
We need to show that
$
\langle f \cdot g,{\mathcal F}\tau\rangle=0$.
Let
$\Phi \in {\mathcal S}({\mathbf R}^n)$ be such that
$\Phi(0)=1$ and that ${\rm supp}({\mathcal F}\Phi) \subset B(1)$.
Then
\[
\langle f \cdot g,{\mathcal F}\tau\rangle
=\lim_{\varepsilon \downarrow 0}
\langle f \cdot g,\Phi(\varepsilon\cdot)^2{\mathcal F}\tau\rangle,
\]
since
\[
\lim_{\varepsilon \downarrow 0}\Phi(\varepsilon\cdot)^2{\mathcal F}\tau
=\tau
\]
in ${\mathcal S}({\mathbf R}^n)$.
If $\varepsilon>0$ is chosen so that
${\rm supp}(\tau) \cap ({\rm supp}({\mathcal F}f)+{\rm supp}({\mathcal F}g)+\overline{B(2\varepsilon)})= \emptyset$,
then we have
\[
\langle \Phi(\varepsilon\cdot)f \cdot \Phi(\varepsilon\cdot)g,{\mathcal F}\tau\rangle=0
\]
thanks to Lemma \ref{lem:180812-3},
since $\Phi(\varepsilon\cdot)f$ and $\Phi(\varepsilon\cdot)g$
are both band-limited
due to Lemma \ref{lem:180812-2}.
Thus,
$
\langle f \cdot g,{\mathcal F}\tau\rangle=0$.
\end{proof}

\subsection{Lipschitz spaces and H\"{o}lder--Zygmund spaces}

Let $0<\alpha \le 1$.
We let ${\rm Lip}^{\alpha}({\mathbf R}^n)$ 
be the set of all bounded continuous functions $f:{\mathbf R}^n \to {\mathbf C}$
for which the quantity
$\displaystyle
\|f\|_{{\rm Lip}^\alpha}
\equiv
\|f\|_{L^\infty}+
\sup_{x,y \in {\mathbf R}^n}|x-y|^{-\alpha}|f(x)-f(y)|
$
is finite.
Let
$\psi$ satisfy
(\ref{eq:180812-3}).
We write
\[
\varphi_0(\xi)=\psi(\xi), \quad
\varphi_j(\xi)=\psi(2^{-j}\xi)-\psi(2^{-j+1}\xi), \quad
\psi_j(\xi)=\psi(2^{-j}\xi)
\]
for $j \in {\mathbf N}$ and $\xi \in {\mathbf R}^n$
as before.
Then the (Besov)--H\"{o}lder--Zygmund space
${\mathcal C}^\beta({\mathbf R}^n)$
with $\beta \in {\mathbf R}$.
is defined to be the set of all
$f \in {\mathcal S}'({\mathbf R}^n)$
for which
\[
\|f\|_{{\mathcal C}^\beta}
=
\sup_{j \in {\mathbf N}_0}2^{j\beta}\|\varphi_j(D)f\|_{L^\infty}
\]
is finite.
Noteworthy is the fact that
${\rm Lip}^\alpha({\mathbf R}^n)$
and
${\mathcal C}^\alpha({\mathbf R}^n)$
are isomorphic
for all $0<\alpha<1$
but that
${\rm Lip}^1({\mathbf R}^n)$
is a proper subset of
${\mathcal C}^1({\mathbf R}^n)$.

Usually we replace
(\ref{eq:180812-3})
by
$\chi_{B(1)} \le \psi \le \chi_{B(2)}$.
However, if we pose a stronger condition
(\ref{eq:180812-3})
on $\psi$,
we can quantify what we are doing.
The following is an example of such an attempt.
\begin{example}
\label{ex:180812-1}
\
Let $j,k,l \in {\mathbf N}$ satisfy $l \ge 2$.
\begin{enumerate}
\item
We note that
$\varphi_k \cdot \psi_{l-2} \ne 0$
only if $l \ge k$.
In this case, we have
\[
{\rm supp}({\mathcal F}[\varphi_k(D)\psi_{l-2}(D)f \cdot \varphi_l(D)g])
\subset
\overline{B\left(\frac32 \cdot 2^k\right)}
+
\overline{B\left(\frac32 \cdot 2^l\right) \setminus B\left(\frac35 \cdot 2^l\right)}.
\]
\item
Assume $l \ge k+2$.
Then
since
\[
\frac18<\frac35-\frac38<\frac32+\frac38<2,
\]
\[
{\rm supp}({\mathcal F}[\varphi_k(D)\psi_{l-2}(D)f \cdot \varphi_l(D)g])
\subset
\overline{B\left(2^{l+1}\right) \setminus B\left(\frac18 \cdot 2^l\right)}.
\]
Consequently,
\[
\varphi_j(D)[\varphi_k(D)\psi_{l-2}(D)f \cdot\varphi_l(D)g] \ne 0
\]
only if 
$l-3 \le j+1 \le l+1$
or
$l-3 \le j-1 \le l+1$,
or equivalently
$l-4 \le j \le l+2$.
\end{enumerate}
\end{example}
\subsection{Some estimates in Besov--Morrey spaces}

For the paraproducts, we use the following observation:
\begin{lemma}\label{lem:180812-4}
Let $1 \le q \le p<\infty$, $1 \le r \le \infty$ and $s \in {\mathbf R}$.
Suppose that we are given
a collection
$\{f_j\}_{j=1}^\infty \subset {\mathcal M}^p_q({\mathbf R}^n)$
satisfying
$f_0 \in {\mathcal S}'_{B(8)}({\mathbf R}^n)$,
$f_j \in {\mathcal S}'_{B(2^{j+3}) \setminus B(2^{j-1})}({\mathbf R}^n)$,
$j \in {\mathbf N}$
and
\[
\left(\sum_{j=0}^\infty (2^{j s}\|f_j\|_{{\mathcal M}^{p}_{q}})^{r}\right)^{\frac1{r}}<\infty.
\]
Then
\[
f=\sum_{j=0}^\infty f_j \in {\mathcal N}^s_{p q r}({\mathbf R}^n)
\]
with
\[
\|f\|_{{\mathcal N}^s_{p q r}}
\le
C
\left(\sum_{j=0}^\infty (2^{j s}\|f_j\|_{{\mathcal M}^{p}_{q}})^{r}\right)^{\frac1{r}}.
\]
\end{lemma}

Let $j \in {\mathbf Z}$ and $\tau \in {\mathcal S}({\mathbf R}^n)$.
Then define
$\tau_j\equiv \tau(2^{-j}\cdot)$.

\begin{proof}
Let $\psi,\varphi_j \in C^\infty_{\rm c}({\mathbf R}^n)$
be as before for each $j \in {\mathbf N}_0$.
$\xi \in {\mathbf R}^n$.
Then
\[
\varphi_k(D)f=\sum_{j=\max(0,k-4)}^{k+4} \varphi_k(D)f_j
\quad (k \in {\mathbf N}).
\]
Thus,
\[
\|\varphi_k(D)f\|_{{\mathcal M}^p_q}
\le
C
\sum_{j=\max(0,k-4)}^{k+4}\|f_j\|_{{\mathcal M}^p_q}
\quad (k \in {\mathbf N}).
\]
As a consequence
\begin{eqnarray*}
\|f\|_{{\mathcal N}^s_{p q r}}
&=&
\left(
\sum_{k=0}^\infty
(2^{k s}
\|\varphi_k(D)f\|_{{\mathcal M}^p_q})^r
\right)^{\frac1r}\\
&\le&
C
\left(
\sum_{k=4}^\infty
\left(2^{k s}
\sum_{j=\max(0,k-4)}^{k+4}
\|f_j\|_{{\mathcal M}^p_q}\right)^r
\right)^{\frac1r}\\
&\le&
C
\sum_{j=0}^8\|f_j\|_{{\mathcal M}^p_q}
+C
\left(
\sum_{k=4}^\infty
\left(2^{k s}
\sum_{l=-4}^{4}
\|f_{k+l}\|_{{\mathcal M}^p_q}\right)^r
\right)^{\frac1r}\\
&\le&
C
\sum_{j=0}^4\|f_j\|_{{\mathcal M}^p_q}
+C
\sum_{l=-4}^{4}
\left(
\sum_{k=4}^\infty
\left(2^{k s}
\|f_{k+l}\|_{{\mathcal M}^p_q}\right)^r
\right)^{\frac1r}\\
&\le&
C
\left(\sum_{j=0}^\infty (2^{j s}\|f_j\|_{{\mathcal M}^{p}_{q}})^{r}\right)^{\frac1{r}},
\end{eqnarray*}
as required.
\end{proof}

\section{Paraproduct}\label{s3}

\subsection{Paraproduct}

For the paraproducts, we use the following observation:
\begin{lemma}\label{lem:180812-5}
Let 
$1 \le q_1 \le p_1<\infty$,
$1 \le r_1 \le \infty$,
$1 \le q_2 \le p_2<\infty$,
$1 \le r_2 \le \infty$,
$1 \le q \le p<\infty$,
$1 \le r \le \infty$,
and
$s_0,s_1,s \in {\mathbf R}$.
Assume that
\[
\frac1p=\frac{1}{p_1}+\frac{1}{p_2}, \quad
\frac1q=\frac{1}{q_1}+\frac{1}{q_2}, \quad
\frac1r=\frac{1}{r_1}+\frac{1}{r_2}, \quad
s=s_1+s_2.
\]
Suppose that we are given
 collections
$\{f_j\}_{j=1}^\infty, \{g_j\}_{j=1}^\infty \subset {\mathcal M}^p_q({\mathbf R}^n)$
satisfying
$f_j \in {\mathcal S}'_{B(2^{j-1})}({\mathbf R}^n)$,
$g_j \in {\mathcal S}'_{B(2^{j+2}) \setminus B(2^{j})}({\mathbf R}^n)$,
$j \in {\mathbf N}$
and
\[
\left(\sum_{j=1}^\infty (2^{j s_1}\|f_j\|_{{\mathcal M}^{p_1}_{q_1}})^{r_1}\right)^{\frac1{r_1}}, \quad
\left(\sum_{j=1}^\infty (2^{j s_2}\|g_j\|_{{\mathcal M}^{p_2}_{q_2}})^{r_2}\right)^{\frac1{r_2}}<\infty.
\]
Then we have
\[
\sum_{j=1}^\infty f_j \cdot g_j \in
{\mathcal N}^s_{p q r}({\mathbf R}^n)
\]
and satisfies
\[
\left\|
\sum_{j=1}^\infty f_j \cdot g_j
\right\|_{{\mathcal N}^s_{p q r}}
\le
C
\left(\sum_{j=1}^\infty (2^{j s_1}\|f_j\|_{{\mathcal M}^{p_1}_{q_1}})^{r_1}\right)^{\frac1{r_1}}
\left(\sum_{j=1}^\infty (2^{j s_2}\|g_j\|_{{\mathcal M}^{p_2}_{q_2}})^{r_2}\right)^{\frac1{r_2}}.
\]
\end{lemma}

\begin{proof}
Thanks to Corollary \ref{cor:180812-1}
we have
$
{\rm supp}(f_j \cdot g_j)
\subset
\overline{B(2^{j+3}) \setminus B(2^{j-1})}
$
for all $j \in {\mathbf N}$.
Thus by the equivalent expression 
(see Lemma \ref{lem:180812-4})
and the H\"{o}lder inequality,
we have
\begin{eqnarray*}
\left\|
\sum_{j=1}^\infty f_j \cdot g_j
\right\|_{{\mathcal N}^s_{p q r}}
&\le&
C
\left(\sum_{j=1}^\infty (2^{j s}\|f_j \cdot g_j\|_{{\mathcal M}^p_q})^r\right)^{\frac1r}
\\
&\le&
C
\left(\sum_{j=1}^\infty (2^{j(s_1+s_2)}\|f_j\|_{{\mathcal M}^{p_1}_{q_1}}\|g_j\|_{{\mathcal M}^{p_2}_{q_2}})^r\right)^{\frac1r}
\\
&\le&
C
\left(\sum_{j=1}^\infty (2^{j s_1}\|f_j\|_{{\mathcal M}^{p_1}_{q_1}})^{r_1}\right)^{\frac1{r_1}}
\left(\sum_{j=1}^\infty (2^{j s_2}\|f_j\|_{{\mathcal M}^{p_2}_{q_2}})^{r_2}\right)^{\frac1{r_2}}.
\end{eqnarray*}
\end{proof}
\subsection{Resonant part}

To handle the resonant part,
we use the following lemma.
When we prove this type of estimates,
we can use the atomic decomposition taking advanatage of the assumption $s>0$
and $p,q,r \ge 1$.
Here we estimate the distributions directly.
This corresponds to \cite[Lemma A3]{GIP15}.
\begin{lemma}\label{lem:180812-6}
Let $1 \le q \le p<\infty$, $1 \le r \le \infty$ and $s>0$.
Suppose that we are given a collection
$\{f_j\}_{j=0}^\infty \subset {\mathcal M}^p_q({\mathbf R}^n)$
satisfying
$f_j \in {\mathcal S}'_{B(2^{j+2})}({\mathbf R}^n)$,
$j \in {\mathbf N}_0$
and
\[
\left(
\sum_{j=0}^\infty (2^{j s}\|f_j\|_{{\mathcal M}^p_q})^r
\right)^{\frac1r}<\infty.
\]
Then
\[
\sum_{j=0}^\infty f_j \in 
{\mathcal N}^s_{p q r}({\mathbf R}^n).
\]
\end{lemma}

\begin{proof}
Let $\psi,\varphi_j \in C^\infty_{\rm c}({\mathbf R}^n)$
be as before for each $j \in {\mathbf N}_0$.
We have
\[
2^{k s}
\sum_{j=0}^\infty
|\varphi_k(D)f_j|
\le
\sum_{j=\max(0,k-3)}^\infty
2^{(k-j)s}|\varphi_k(D)[2^{j s}f_j]|.
\]
As a consequence,
by the translation invariance of 
${\mathcal M}^p_q({\mathbf R}^n)$
and the equality
$\|{\mathcal F}^{-1}\varphi_k\|_{L^1}=\|{\mathcal F}^{-1}\varphi_1\|_{L^1}$
for all $k \in {\mathbf N}$
\begin{eqnarray*}
\left\|
2^{k s}
\sum_{j=0}^\infty
|\varphi_k(D)f_j|
\right\|_{{\mathcal M}^p_q}
&\le&
C
\sum_{j=\max(0,k-3)}^\infty
2^{(k-j)s}(2^{j s}\|\varphi_k(D)f_j\|_{{\mathcal M}^p_q})\\
&\le&
C
\sum_{j=\max(0,k-3)}^\infty
2^{(k-j)s}(2^{j s}\|f_j\|_{{\mathcal M}^p_q}).
\end{eqnarray*}
Since $s>0$,
by the H\"{o}lder inequality
\begin{eqnarray*}
&&\left\|
2^{k s}
\sum_{j=0}^\infty
|\varphi_k(D)f_j|
\right\|_{{\mathcal M}^p_q}\\
&\le&
C
\sum_{j=\max(0,k-3)}^\infty
2^{\frac12(k-j)s}2^{\frac12(k-j)s}(2^{j s}\|f_j\|_{{\mathcal M}^p_q})\\
&\le&
C
\left(\sum_{j=\max(0,k-3)}^\infty
2^{\frac12(k-j)s r'}\right)^{\frac{1}{r'}}
\left(\sum_{j=\max(0,k-3)}^\infty
(2^{\frac12(k-j)s}(2^{j s}\|f_j\|_{{\mathcal M}^p_q}))^r\right)^{\frac1r}\\
&\le&
C
\left(\sum_{j=\max(0,k-3)}^\infty
(2^{\frac12(k-j)s}(2^{j s}\|f_j\|_{{\mathcal M}^p_q}))^r\right)^{\frac1r}.
\end{eqnarray*}

Thus,
if we take the $\ell^r$-norm,
then we obtain
\[
\left\|\sum_{j=0}^\infty f_j\right\|_{{\mathcal N}^s_{p q r}}
\le
C
\left(
\sum_{j=0}^\infty (2^{j s}\|f_j\|_{{\mathcal M}^p_q})^r
\right)^{\frac1r}.
\]
\end{proof}

\begin{corollary}\label{cor:180812-2}
Let 
$1 \le q_1 \le p_1<\infty$,
$1 \le r_1 \le \infty$,
$1 \le q_2 \le p_2<\infty$,
$1 \le r_2 \le \infty$,
$1 \le q \le p<\infty$,
$1 \le r \le \infty$,
and
$s_0,s_1,s \in {\mathbf R}$.
Assume that
\[
\frac1p=\frac{1}{p_1}+\frac{1}{p_2}, \quad
\frac1q=\frac{1}{q_1}+\frac{1}{q_2}, \quad
\frac1r=\frac{1}{r_1}+\frac{1}{r_2}, \quad
s=s_1+s_2>0.
\]
Suppose that we are given collections
$\{f_j\}_{j=0}^\infty, \{g_j\}_{j=0}^\infty \subset {\mathcal M}^p_q({\mathbf R}^n)$
satisfying
$f_j, g_j \in {\mathcal S}'_{B(2^{j+1})}({\mathbf R}^n)$,
$j \in {\mathbf N}_0$
and
\[
\left(
\sum_{j=0}^\infty (2^{j s_1}\|f_j\|_{{\mathcal M}^{p_1}_{q_1}})^{r_1}
\right)^{\frac1{r_1}}, \quad
\left(
\sum_{j=0}^\infty (2^{j s_2}\|g_j\|_{{\mathcal M}^{p_2}_{q_2}})^{r_2}
\right)^{\frac1{r_2}}<\infty
\]
Then
\[
\sum_{j=0}^\infty f_j \cdot g_j \in 
{\mathcal N}^s_{p q r}({\mathbf R}^n).
\]
and
\[
\left\|\sum_{j=0}^\infty f_j \cdot g_j\right\|_{{\mathcal N}^s_{p q r}}
\le
C
\left(
\sum_{j=0}^\infty (2^{j s_1}\|f_j\|_{{\mathcal M}^{p_1}_{q_1}})^{r_1}
\right)^{\frac1{r_1}}
\left(
\sum_{j=0}^\infty (2^{j s_2}\|g_j\|_{{\mathcal M}^{p_2}_{q_2}})^{r_2}
\right)^{\frac1{r_2}}.
\]
\end{corollary}

\begin{proof}
In fact,
by Corollary \ref{cor:180812-1},
we see that $f_j \cdot g_j \in {\mathcal S}'_{B(2^{j+2})}({\mathbf R}^n)$.
Thus,
invoking Lemma \ref{lem:180812-6} and using the H\"{o}lder inequality twice,
we have
\begin{eqnarray*}
\left\|\sum_{j=0}^\infty f_j \cdot g_j\right\|_{{\mathcal N}^s_{p q r}}
&\le& C
\left(\sum_{j=0}^\infty (2^{j s}\|f_j \cdot g_j\|_{{\mathcal M}^p_q})^r\right)^{\frac1r}\\
&\le&
C
\left(
\sum_{j=0}^\infty (2^{j s_1}\|f_j\|_{{\mathcal M}^{p_1}_{q_1}})^{r_1}
\right)^{\frac1{r_1}}
\left(
\sum_{j=0}^\infty (2^{j s_2}\|g_j\|_{{\mathcal M}^{p_2}_{q_2}})^{r_2}
\right)^{\frac1{r_2}}.
\end{eqnarray*}
\end{proof}

\subsection{Conclusion of the proof of Theorem \ref{thm:main1}}

We prove Theorem \ref{thm:main1} as follows:
If we use Lemma \ref{lem:180812-5},
then we have
\begin{eqnarray*}
\|f \preceq g\|_{{\mathcal N}^{s}_{p q r}}
&=&
\left\|\sum_{j=2}^\infty \psi_{j-2}(D)f \cdot \varphi_j(D)g\right\|_{{\mathcal N}^{s}_{p q r}}\\
&\le& C
\sup_{j \in {\mathbf N}_0}
\|\psi_{j}(D)f\|_{{\mathcal M}^{p_1}_{q_1}}
\left(
\sum_{j=2}^\infty 
(2^{j s}\|g_j\|_{{\mathcal M}^{p_2}_{q_2}})^r
\right)^{\frac1r}.
\end{eqnarray*}
Since $\psi_{j}(D)f=(2\pi)^{-\frac{n}{2}}{\mathcal F}^{-1}\psi_j*f$
and ${\mathcal F}^{-1}\psi_j=2^{j n}{\mathcal F}^{-1}\psi(2^j \cdot)$,
we have
\begin{eqnarray*}
\|\psi_{j}(D)f\|_{{\mathcal M}^{p_1}_{q_1}}
&=&(2\pi)^{-\frac{n}{2}}\|{\mathcal F}^{-1}\psi_j*f\|_{{\mathcal M}^{p_1}_{q_1}}\\
&\le&(2\pi)^{-\frac{n}{2}}\|{\mathcal F}^{-1}\psi_j\|_{L^1}\|f\|_{{\mathcal M}^{p_1}_{q_1}} \\
&=&(2\pi)^{-\frac{n}{2}}\|{\mathcal F}^{-1}\psi\|_{L^1}\|f\|_{{\mathcal M}^{p_1}_{q_1}}.
\end{eqnarray*}
Thus,
\begin{eqnarray*}
\|f \preceq g\|_{{\mathcal N}^{s}_{p q r}}
&\le& C
\|f\|_{{\mathcal M}^{p_1}_{q_1}}
\left(
\sum_{j=2}^\infty 
(2^{j s}\|g_j\|_{{\mathcal M}^{p_2}_{q_2}})^r
\right)^{\frac1r}.
\end{eqnarray*}
Recall that $s>0$.
Since
${\mathcal M}^{p_1}_{q_1}({\mathbf R}^n)
\supset {\mathcal N}^{s}_{p_1q_1r}({\mathbf R}^n)$,
we have
\begin{eqnarray*}
\|f \preceq g\|_{{\mathcal N}^{s}_{p q r}}
&\le& C
\|f\|_{{\mathcal N}^s_{p_1q_1r}}
\|g\|_{{\mathcal N}^s_{p_2q_2r}}.
\end{eqnarray*}
Likewise
\begin{eqnarray*}
\|f \succeq g\|_{{\mathcal N}^{s}_{p q r}}
&\le& C
\|f\|_{{\mathcal N}^s_{p_1q_1r}}
\|g\|_{{\mathcal N}^s_{p_2q_2r}}.
\end{eqnarray*}

Meanwhile,
we have
\begin{eqnarray*}
\|f \odot g\|_{{\mathcal N}^{s}_{p q r}}
\le
\|f \odot g\|_{{\mathcal N}^{2s}_{p q r}}
\le C
\|f\|_{{\mathcal N}^{s}_{p_1 q_1 2r}}
\|g\|_{{\mathcal N}^{s}_{p_2 q_2 2r}}
\le C
\|f\|_{{\mathcal N}^{s}_{p_1 q_1 r}}
\|g\|_{{\mathcal N}^{s}_{p_2 q_2 r}}
\end{eqnarray*}
by Corollary \ref{cor:180812-2}.

Putting together these observations, we obtain the desired result.

\section{Commutator estimate}\label{s4}

We recall the following lemma obtained in \cite[Lemma 2.2]{GIP15}:
\begin{lemma}\label{lem:180812-1}
Let $0<\alpha \le 1$, $j \in {\mathbf N}_0$,
and let
$F \in {\rm Lip}^\alpha({\mathbf R}^n)$,
$G \in L^\infty({\mathbf R}^n)$.
Then
\[
\|\varphi_j(D)[F \cdot G]-F \varphi_j(D)G\|_{L^\infty}
\le C
2^{-j\alpha}
\|F\|_{{\rm Lip}^\alpha}
\|G\|_{L^\infty}.
\]
\end{lemma}

This is a slight extension
of \cite[Lemma 2.2]{GIP15} to the case where $\alpha=1$.
Here for the sake of convenience for readers,
we recall the whole proof.

\begin{proof}
Since $\varphi_j(D)H(x)=(2\pi)^{-\frac{n}{2}}{\mathcal F}^{-1}\varphi_j*H(x)$
for all $H \in {\mathcal S}'({\mathbf R}^n)$
which grows polynomially at infinity,
\begin{eqnarray*}
\lefteqn{
\varphi_j(D)[F \cdot G](x)-F(x) \varphi_j(D)G(x)
}\\
&=&(2\pi)^{-\frac{n}{2}}
\int_{{\mathbf R}^n}
2^{j n}{\mathcal F}^{-1}\varphi(2^j(x-y))(F(y)-F(x))G(y)\,dy.
\end{eqnarray*}
As a result,
letting
\[
C=(2\pi)^{-\frac{n}{2}}\int_{{\mathbf R}^n}
|z|^\alpha|{\mathcal F}^{-1}\varphi(z)|\,dz,
\]
we have
\[
\|\varphi_j(D)[F \cdot G]-F \varphi_j(D)G\|_{L^\infty}
\le C
2^{-j\alpha}
\|F\|_{{\rm Lip}^\alpha}
\|G\|_{L^\infty},
\]
as required.
\end{proof}

\begin{lemma}\label{lem:180812-7}
Let $0<\alpha \le 1$, $j \in {\mathbf N}_0$,
and let
$F \in {\rm Lip}^\alpha({\mathbf R}^n)$,
$G \in L^\infty({\mathbf R}^n)$.
Then we have
\[
\|\varphi_j(D)[F \preceq G]-F \varphi_j(D)G\|_{L^\infty}
\le C
2^{-j\alpha}\|F\|_{{\rm Lip}^\alpha}\|G\|_{L^\infty}.
\]
\end{lemma}

This is also a slight extension
of \cite[Lemma 2.3]{GIP15} to the case where $\alpha=1$.
Here for the sake of convenience for the readers
we supply the proof.

\begin{proof}
We assume $j \gg 1$;
otherwise we can mimic the argument below
and we can readily incorporate the case where $j$ is not so large.
We decompose
\begin{eqnarray*}
\lefteqn{
\varphi_j(D)[F \preceq G]-F \varphi_j(D)G
}\\
&=&
\sum_{k=j-3}^{j+3}\left(\varphi_j(D)[F \preceq \varphi_k(D)G]-F \varphi_j(D)\varphi_k(D)G\right)\\
&=&
\sum_{k=j-3}^{j+3}\left(\varphi_j(D)[F \cdot \varphi_k(D)G]-F \varphi_j(D)\varphi_k(D)G
-\varphi_j(D)[F \succeq \varphi_k(D)G]\right).
\end{eqnarray*}
Let $k$ be fixed.
We use Lemma \ref{lem:180812-1} to have
\[
\|\varphi_j(D)[F \cdot \varphi_k(D)G]-F \varphi_j(D)\varphi_k(D)G\|_{L^\infty}
\le C2^{-j\alpha}
\|F\|_{{\rm Lip}^\alpha}\|G\|_{L^\infty}.
\]
Meanwhile, using
\[
\varphi_j(D)[F \succeq \varphi_k(D)G]
=
\sum_{l=j-5}^{j+5}\varphi_j(D)[\varphi_l(D)F \succeq \varphi_k(D)G]
\]
for $k \in [j-3,j+3]$,
we have
\begin{eqnarray*}
\|\varphi_j(D)[F \succeq \varphi_k(D)G]\|_{L^\infty}
&\le&
\sum_{l=j-5}^{j+5}\|\varphi_j(D)[\varphi_l(D)F \succeq \varphi_k(D)G]\|_{L^\infty}\\
&\le& C
\sum_{l=j-5}^{j+5}\|\varphi_l(D)F \succeq \varphi_k(D)G\|_{L^\infty}\\
&\le& C\sup_{l,l' \in {\mathbf N}_0}
\|\varphi_l(D)F \cdot \varphi_{l'}(D)G\|_{L^\infty}\\
&\le& C
2^{-j\alpha}\|F\|_{{\rm Lip}^\alpha}\|G\|_{L^\infty}.
\end{eqnarray*}
\end{proof}

We prove Theorem \ref{thm:main2}
to conclude this note.

\begin{proof}
We decompose
\begin{eqnarray*}
(f \preceq g) \odot h-f(g \odot h)
&=&
\sum_{j=0}^\infty 
[\varphi_j(D)[f \preceq g]
-
f\cdot\varphi_j(D)g] \varphi_j(D)h\\
&&+
\sum_{j=1}^\infty 
(\varphi_{j-1}(D)[f \preceq g] -
f\varphi_{j-1}(D)g) \varphi_j(D)h\\
&&+
\sum_{j=1}^\infty 
[\varphi_j(D)[f \preceq g]
-
f\cdot\varphi_j(D)g] \varphi_{j-1}(D)h.
\end{eqnarray*}
We handle the first term;
other two terms are dealt with similarly.
We decompose
\begin{eqnarray*}
\lefteqn{
\sum_{j=0}^\infty 
[\varphi_j(D)[f \preceq g]
-
f\cdot\varphi_j(D)g] \varphi_j(D)h
}\\
&=&
\sum_{j=0}^\infty 
[\varphi_j(D)[\psi_{j+4}(D)f \preceq g]
-
\psi_{j+4}(D)f\cdot\varphi_j(D)g] \varphi_j(D)h\\
&&+\sum_{j=0}^\infty \sum_{k=j+5}^\infty
\varphi_j(D)[\varphi_k(D)f \preceq g] \cdot \varphi_j(D)h\\
&&-\sum_{j=0}^\infty \sum_{k=j+5}^\infty
\varphi_k(D)f\cdot\varphi_j(D)g \cdot \varphi_j(D)h.
\end{eqnarray*}
Since
\[
\|\partial^m[\varphi_j(D)[\psi_{j+4}(D)f \preceq g]
-
\psi_{j+4}(D)f\cdot\varphi_j(D)g]\|_{L^\infty}
=
{\rm O}(2^{-j(\alpha+\beta-|m|)})
\]
for all $m=(m_1,m_2,\ldots,m_n) \in {\mathbf N}_0{}^n$,
we have
\begin{eqnarray*}
&&\left\|\sum_{j=0}^\infty 
[\varphi_j(D)[\psi_{j+4}(D)f \preceq g]
-
\psi_{j+4}(D)f\cdot\varphi_j(D)g] \varphi_j(D)h\right\|_{{\mathcal N}^{s+\alpha+\beta}_{p q r}}
\\
&&\le C
\|f\|_{{\rm Lip}^\alpha}
\|g\|_{{\mathcal C}^\beta}
\|h\|_{{\mathcal N}^s_{p q r}}.
\end{eqnarray*}

Using Example \ref{ex:180812-1},
we estimate
the second term:
\begin{eqnarray*}
&&\sum_{j=0}^\infty \sum_{k=j+5}^\infty
\varphi_j(D)[\varphi_k(D)f \preceq g] \cdot \varphi_j(D)h\\
&&=
\sum_{j=0}^\infty \sum_{k=j+5}^\infty
\varphi_j(D)[\varphi_k(D)\psi_{k-2}(D)f \cdot\varphi_k(D)g] \cdot \varphi_j(D)h\\
&&\quad+
\sum_{j=0}^\infty \sum_{k=j+5}^\infty
\varphi_j(D)[\varphi_k(D)\psi_{k-1}(D)f \cdot\varphi_{k+1}(D)g] \cdot \varphi_j(D)h.
\end{eqnarray*}
Next, we note that
\begin{eqnarray*}
&&\|\varphi_k(D)f \cdot \varphi_j(D)g \cdot \varphi_j(D)h\|_{{\mathcal M}^p_q}\\
&&\le C
2^{-k\alpha-j(s+\beta)}
\|f\|_{{\rm Lip}^\alpha}
\|g\|_{{\mathcal C}^\beta}
\|2^{j s}\varphi_j(D)h\|_{{\mathcal M}^p_q}.
\end{eqnarray*}
Adding this estimate over $j,k$,
we have
\begin{eqnarray*}
&&
\left\{
\sum_{k=5}^\infty
\left(2^{k(s+\alpha+\beta)}\left\|\sum_{j=0}^{k-5}\varphi_k(D)f \cdot \varphi_j(D)g \cdot \varphi_j(D)h\right\|_{{\mathcal M}^p_q}\right)^r
\right\}^{\frac1r}\\
&&\le C
\left\{
\sum_{k=5}^\infty
\left(
\sum_{j=0}^{k-5}2^{(k-j)(s+\beta)}
\|f\|_{{\rm Lip}^\alpha}
\|g\|_{{\mathcal C}^\beta}
\|2^{j s}\varphi_j(D)h\|_{{\mathcal M}^p_q}
\right)^r
\right\}^{\frac1r}\\
&&= C
\left\{
\sum_{k=5}^\infty
\sum_{j=0}^{k-5}
\left(
2^{\frac12(k-j)(s+\beta)}
\|f\|_{{\rm Lip}^\alpha}
\|g\|_{{\mathcal C}^\beta}
\|2^{j s}\varphi_j(D)h\|_{{\mathcal M}^p_q}
\right)^r
\right\}^{\frac1r}\\
&&= C
\left\{
\sum_{j=0}^\infty
\left(
\|f\|_{{\rm Lip}^\alpha}
\|g\|_{{\mathcal C}^\beta}
\|2^{j s}\varphi_j(D)h\|_{{\mathcal M}^p_q}
\right)^r
\right\}^{\frac1r}\\
&&=C\|f\|_{{\rm Lip}^\alpha}
\|g\|_{{\mathcal C}^\beta}
\|h\|_{{\mathcal N}^s_{p q r}}.
\end{eqnarray*}
\end{proof}

\section{Acknowledgement}

The author thankful to Professor Alexey Karapetyants
for his inviting me to the conference OTHA 2018.
The author is also thankful
to Professors Yuzuru Inahama and Masato Hoshino
for their encouragement to write this note.

%
%

Department of Mathematical Science,\\
1-1 Minami-Ohsawa, Hachioji, 192-0397, Tokyo Japan\\
\email{yoshihiro-sawano@celery.ocn.ne.jp}
\end{document}